\documentclass[a4paper,11pt]{amsart}
\usepackage{amsthm, amsmath, amssymb,amsbsy,graphicx,mathrsfs,enumerate, bm}
\usepackage[vcentermath, enableskew]{youngtab}
\usepackage[margin=3cm]{geometry}

\def\wid1{15.7cm}

\numberwithin{table}{section}

\newlength\savedwidth

\usepackage{caption}
\DeclareCaptionLabelSeparator{period}{. }
\newtheorem{thm}{Theorem}
\newtheorem{lem}[thm]{Lemma}
\newtheorem{cor}[thm]{Corollary}

\newtheorem{defi}[thm]{Definition}

\theoremstyle{definition}
\newtheorem{remark}[thm]{Remark}
\newtheorem{exam}{Example}[section]

\newcommand{\mZ}{\mathbb{Z}}
\newcommand{\mF}{\mathbb{F}}

\newcommand{\mN}{\mathbb{N}}

\renewcommand{\rm}{\mathrm}

\newcommand{\mbf}{\mathbf}

\newcommand{\thra}{\twoheadrightarrow}

\newcommand{\sm}{\setminus}

\newcommand{\lda}{\lambda}

\newcommand{\fr}{\mathfrak}

\newcommand{\scr}{\mathscr}

\newcommand{\ol}{\overline}

\renewcommand{\a}{\alpha}
\renewcommand{\b}{\beta}
\newcommand{\e}{\epsilon}
\newcommand{\g}{\gamma}
\renewcommand{\d}{\delta}
\renewcommand{\k}{\kappa}

\renewcommand{\fr}{\mathfrak}

\newcommand{\lan}{\langle}
\newcommand{\ran}{\rangle}

\newcommand{\Char}{\operatorname{char}}

\newcommand{\Lda}{\Lambda}

\newcommand{\cl}{\mathcal}

\newcommand{\mtt}{\mathtt}

\newcommand{\qdim}{\operatorname{qdim}}

\title[On graded decomposition numbers]{On graded decomposition numbers for cyclotomic Hecke algebras in quantum characteristic 2}
\author{Anton Evseev}
\address{School of Mathematics, University of Birmingham, Edgbaston, Birmingham B15 2TT, UK}
\email{a.evseev@bham.ac.uk}

\begin{document}

\begin{abstract}
 Brundan and Kleshchev introduced graded decomposition numbers for representations of cyclotomic Hecke algebras of type $A$, which include group algebras of symmetric groups. 
Graded decomposition numbers are certain Laurent polynomials, whose values at $1$ are the usual decomposition numbers. 
We show that in quantum characteristic $2$ every such polynomial has non-zero coefficients either only in odd or only in even degrees. As a consequence, we find the first examples of graded decomposition numbers of symmetric groups with non-zero coefficients in some negative degrees.
\end{abstract}

\maketitle

\section{Introduction}\label{sec:intro}

The problem of finding decomposition numbers for representations of a symmetric group $\fr S_d$ in positive characteristic $p$ is old and, despite important advances, largely unresolved. The decomposition numbers are the nonnegative integers $[S(\mu): D(\lda)]$ that are multiplicities of simple modules $D(\lda)$ in composition series of Specht modules $S(\mu)$, where $\mu$ is a partition of $d$ and $\lda$ is a  $p$-restricted partition of $d$. 
In 2009, Brundan and Kleshchev~\cite{BrundanKleshchev2009a} found a natural $\mZ$-grading on the group algebra $\mF_p \fr S_d$, which comes from an isomorphism of that algebra with a quotient of a KLR algebra defined by Khovanov and Lauda~\cite{KhovanovLauda2009} and independently by Rouquier~\cite{Rouquier2008}. By~\cite{BrundanKleshchev2009, BrundanKleshchevWang2011}, both Specht modules and simple modules of $\mF_p \fr S_d$ are naturally graded, and therefore one may consider graded decomposition numbers $[S(\mu): D(\lda)]_q \in \mN[q,q^{-1}]$. The usual decomposition numbers $[S(\mu): D(\lda)]$ are then obtained by specialising $q$ to $1$. Clearly, these discoveries add a lot more structure to the problem. 

In this note we focus on the case $p=2$. Under this assumption, 
we show that each graded module $S(\mu)$ 
has non-zero components either only in odd or only in even degrees (Theorem~\ref{thm:1}). We deduce that each Laurent polynomial $[S(\mu):D(\lda)]_q$ has non-zero coefficients either only in odd or only in even powers of $q$
(Corollary~\ref{cor:1}). In both cases, the relevant parities are given by simple combinatorial rules. As a consequence, we give some new information on ungraded decomposition numbers (Corollary~\ref{cor:ungr}). All these results extend to the more general setting of cyclotomic Hecke algebras. 

It has been suggested that all graded decomposition numbers of $\mF_p \fr S_d$ might have zero coefficients in negative powers of $q$, which is equivalent to saying that graded adjustment matrices (defined by Eq.~\eqref{eq:dec} below) have integer entries: cf.\ remarks in~\cite[\S 5.6]{BrundanKleshchev2009} and~\cite[\S 10.3]{Kleshchev2010}. Using the aforementioned results, we find examples showing that this is not the case when $p=2$ (Corollary~\ref{cor:notinZ}). 

Section~\ref{sec:prelim} is a brief review of definitions and results about graded cyclotomic Hecke algebras $H_d^{\Lda}$ and their Specht modules. Generally, we follow the notation of~\cite{Kleshchev2010}, which is a much more detailed survey of the subject. 
The new results are in Section~\ref{sec:results}.

\section{Graded cyclotomic Hecke algebras}\label{sec:prelim}

The set of nonnegative integers is denoted by $\mN$. 
The arguments of this note work only in quantum characteristic $2$, so we specialise the  definitions to that case. 
Set $I=\mZ/2\mZ = \{0,1\}$. We will consider cyclotomic KLR algebras associated with the quiver
\[
 0 \leftrightarrows 1.
\]
The corresponding Cartan matrix is the $I\times I$-matrix 
\[
 A= (a_{ij})_{i,j\in I} = 
\begin{pmatrix}
  2 & -2 \\
  -2 & 2
 \end{pmatrix}.
\]
Consider a realisation of $A$ (see~\cite[\S 1.1]{Kac1990}) with the simple coroots 
$\a_0^{\vee}$ and $\a_1^{\vee}$ and 
weights $\Lda_0$ and $\Lda_1$ such that $(\a_i^{\vee}, \Lda_j)=\d_{ij}$. 
Fix $l\in \mN$ and an $l$-tuple
$\k = (k_1,\ldots,k_l) \in I^l$. Set $\Lda = \Lda_{k_1} + \cdots + \Lda_{k_l}$. 

Let $F$ be an arbitrary field and fix $d\in \mN$.
The cyclotomic Khovanov--Lauda--Rouquier (KLR) algebra $H_d^{\Lda}$ is the $F$-algebra defined by the generators 
\begin{equation}\label{eq:gens}
 \{ e(\mbf i) \mid i \in I^d \} \cup \{ y_1,\ldots, y_d \} \cup \{\psi_1,\ldots,\psi_{d-1} \}
\end{equation}
subject to a lengthy list of relations given in~\cite{BrundanKleshchev2009a} (see also, for example,~\cite{Kleshchev2010}). In particular, the relations ensure that the elements $e(\mbf i)$, $\mbf i\in I^d$, are pairwise orthogonal idempotents summing to $1$. 
The algebra $H^{\Lda}_d$ is the quotient of a KLR algebra (see~\cite{KhovanovLauda2009, Rouquier2008}) by a ``cyclotomic relation''. 

By~\cite[Corollary 1]{BrundanKleshchev2009a}, 
$H_d^{\Lda}$ may be graded as follows: $\deg(e(\mbf i))=0$, $\deg(y_r)=2$, 
\[
\deg(\psi_r e(\mbf i)) = -a_{i_r,i_{r+1}} = \begin{cases}
                                             -2 & \text{if } i_r=i_{r+1}; \\
					      2 & \text{if } i_r\ne i_{r+1}. 
                                            \end{cases}
\]

It is proved in~\cite{BrundanKleshchev2009a} that
$H_d^{\Lda}$ is isomorphic to the corresponding cyclotomic Hecke algebra (also denoted by $H_d^{\Lda}$), which means that the latter is naturally graded.
In particular, $H_d^{\Lda_0}$ is isomorphic to $F \fr S_d$ if $\Char F = 2$ and to the usual Iwahori--Hecke algebra $\cl H_d (-1)$ if $\Char F =0$.

We now consider graded Specht modules, which were first constructed in~\cite{BrundanKleshchevWang2011} and may also be defined as the cell modules of the graded cellular algebra structure on $H_d^{\Lda}$ that was found in~\cite{HuMathas2010}. The Specht modules $S(\lda)$ of $H^{\Lda}_d$ are parameterised by \emph{$l$-multipartitions} of size $d$, that is, by $l$-tuples 
$\lda= (\lda^{(1)},\ldots,\lda^{(l)} )$ of partitions such that $\sum_m |\lda^{(m)}| = d$. 
The set of such $l$-multipartitions is denoted by $\scr{P}^{\k}_d$.  
The Young diagram of an $l$-multipartition $\lda$ is the set
\[
 Y(\lda) = \{ (a,b,m) \in \mZ_{>0} \times \mZ_{>0} \times \{ 1,\ldots, l \} \mid 1\le b\le \lda_a^{(m)} \},
\]
where $\lda_a^{(m)}$ is the $a$-th part of $\lda^{(m)}$. 
The elements of $Y(\lda)$ are called the \emph{nodes} of $\lda$, and more generally a \emph{node} is an element of $\mZ_{>0} \times \mZ_{>0} \times \{1,\ldots, l\}$. 

A $\lda$-tableau is a bijective assignment of the integers $1,\ldots,d$ to the nodes of $\lda$. A $\lda$-tableau $\mtt t$ is said to be \emph{standard} if, whenever $1\le m\le l$ and $(a,b,m)$, $(a',b',m)$ are nodes of $\lda$ such that $a\le a'$ and $b\le b'$, the entry of $\mtt t$ in $(a,b,m)$ is less than or equal to the one in $(a',b',m)$. 
The set of all standard $\lda$-tableaux is denoted by $\scr T(\lda)$. 
The \emph{residue} of a node $(a,b,m)$ of a tableau $\mtt t$ is defined to be
$k_m + (b-a) +2\mZ \in I$; a node of residue $i$ is said to be an \emph{$i$-node}.
The \emph{residue sequence} of $\mtt t$ is the sequence $\mbf i =(i_1,\ldots, i_d)\in I^r$
where $i_r$ is the residue of the node occupied by $r$ in $\mtt t$. 

We say that a node $N\in Y(\lda)$ is \emph{removable} if $Y(\lda) \sm \{ N\}$ is the Young diagram of a multipartition; a node $N'\notin Y(\lda)$ is \emph{addable} for $\lda$ if $Y(\lda) \cup \{ N' \}$ is the Young diagram of a multipartition. 
If $(a,b,m)$ and $(a',b',m')$ are two nodes, we say that the latter is 
\emph{below} the former if either $m'>m$ or $m'=m$ and $a'>a$. 

Let $\mtt t$ be a standard tableau and let $N$ be the node occupied by $d$ in $\mtt t$. Define
\[
 d_N (\lda) = \# \{ \text{addable } i\text{-nodes for } \lda \text{ below } N \} 
- \# \{ \text{removable } i\text{-nodes of } \lda \text{ below } N \}.
\]
The \emph{degree} of the standard tableau $\mtt t$ is defined recursively by 
\[
 \deg(\mtt t) = d_N (\lda) + \deg(\mtt t_{\le (d-1)}),
\]
where $t_{\le r}$ denotes the standard tableau consisting of the boxes occupied by $1,\ldots,r$ in $\mtt t$  and the degree of the empty multipartition $( \varnothing, \ldots,\varnothing )$ is set to be $0$. 
It is proved in~\cite{BrundanKleshchevWang2011} 
that the Specht module $S(\lda)$ is naturally graded and has a certain homogeneous basis
$\{ z_{\mtt t} \}_{\mtt t\in \scr{T}(\lda) }$ with the degree of
$z_{\mtt t}$ equal to $\deg (\mtt t)$.

If $M = \oplus_{n\in \mZ} M_n$ is the decomposition into components of a finite-dimensional graded vector space $M$, the graded dimension of $M$ is defined to be the polynomial
\[
 \qdim (M) = \sum_{n\in \mZ} (\dim M_n) q^n \in \mZ [q, q^{-1}].
\]
The aforementioned facts give a combinatorial description of $\qdim S(\lda)$: 
\begin{equation}\label{eq:qdimSlda}
 \qdim S(\lda) = \sum_{\mtt t \in \scr{T}(\lda)} q^{\deg(\mtt t)}
\end{equation}
for any multipartition $\lda$.

\section{Parity results}\label{sec:results}

Our starting point is the observation that, while the odd-degree components of $H_d^{\Lda}$ are all zero, this is not necessarily the case for the Specht modules $S(\lda)$. To give more detail, we need to introduce some further notation. 

We are mainly concerned with the \emph{parity} of degrees, thus projecting $\mZ$-gradings to $\mZ/2\mZ$-gradings.
Let $S$ be the ring $\mZ[q,q^{-1}]/(q^2-1)\mZ[q,q^{-1}]$, and let $u=q+(q^2-1)\mZ[q,q^{-1}]\in S$. 
As an abelian group, $S=\mZ 1 \oplus \mZ u$, and we identify $\mZ$ with the subring $\mZ1$ of $S$. 
We have a map $\mZ[q,q^{-1}] \thra S$ given by $q\mapsto u$.
The image of any $f\in \mZ[q,q^{-1}]$ is denoted by $f(u)$. 
The fact that the odd-degree components of $H_d^{\Lda}$ are $0$ may then be expressed as follows:
\begin{equation}\label{eq:algeven}
(\qdim H^{\Lda}_d) (u) \in \mN.
\end{equation}

For a partition $\lda=(\lda_1,\ldots,\lda_n)$, define its \emph{parity} $\e(\lda)\in \mZ/2\mZ$ by
\[
 \e(\lda) = \sum_{i=1}^n \left\lfloor \frac{\lda_i}{2} \right\rfloor + 2\mZ, 
\]
where $\lfloor \g \rfloor$ denotes the integer part of a rational number $\g$. 
For $i\in I$, let $n_i (\lda)$ denote the number of $i$-nodes in the Young diagram of $\lda$.
 For a multipartition 
$\lda=(\lda^{(1)}, \ldots,\lda^{(l)} )$, define 
\begin{equation}\label{eq:defeps}
\e(\lda) = \e(\lda^{(1)}) + \cdots + \e(\lda^{(l)}) + \sum_{j=1}^{l-1} \sum_{m=j+1}^{l} n_{k_m} (\lda^{(j)}) + 2\mZ,
\end{equation}
where $\k = (k_1,\ldots,k_l) \in I^l$ is the tuple fixed in Section~\ref{sec:prelim}. 
In the case when $l=1$, the last term is $0$, so the two definitions agree.


For any $\lda\in \scr{P}_d^{\k}$, let $\mtt t^{\lda}$ be the standard $\lda$-tableau in which the numbers $1,\ldots,d$ are filled along successive rows of $\lda$, from top to bottom.

\begin{lem}\label{lem:tlda}
 For any $\lda \in \scr{P}_d^{\k}$, we have $\e(\lda) = \deg(\mtt t^{\lda}) + 2\mZ$. 
\end{lem}

\begin{proof}
 Let $r\in \{1,\ldots, d\}$, and let $N=(a,b,j)$ be the node occupied by $r$ in $\mtt t^{\lda}$. Let $i$ be the residue of $N$. 
Then  $\mtt t^{\lda}_{\le r}$ has no removable nodes below $N$. The addable nodes below $N$ for $\mtt t^{\lda}_{\le r}$ are $(a+1,1,j)$ and the nodes $(1,1,m)$ for $j<m\le l$. 
Hence, $d_N (t^{\lda}_{\le r} )=\b + \sum_{m=j+1}^l \d_{k_m, i}$
where
\[
 \b=\begin{cases}
     1 & \text{if } b \text{ is even}, \\
     0 & \text{if } b \text{ is odd}. 
    \end{cases}
\] 
Summing over all $j\in \{1,\ldots, d\}$, we see that $\deg(\mtt t^{\lda})$ modulo $2$ is given by the expression on the right-hand side of~\eqref{eq:defeps}. 
\end{proof}

Note that $u^m\in S$ is well-defined for any $m\in \mZ/2\mZ$.

\begin{thm}\label{thm:1}
 For any $\lda \in \scr{P}_d^{\k}$,
\[
 \qdim S(\lda) (u)  \in \mN u^{\e(\lda)}. 
\]
\end{thm}

\begin{proof}
Let $\mtt t$ be any standard $\lda$-tableau. Let $w_{\mtt t} \in \fr S_d$ be the element such that 
$w_{\mtt t} (\mtt t^{\lda})= \mtt t$, where $\fr S_d$ acts on $\scr{T}(\lda)$ by permuting the entries $\{1,\ldots, d\}$. 
Let $w_{\mtt t} = s_{r_1} \cdots s_{r_m}$ be a reduced decomposition of $w_{\mtt t}$ with respect to the Coxeter generators $s_r =(r,r+1)$ of $\fr S_d$.  
By~\cite[Corollary 3.14]{BrundanKleshchevWang2011}, 
\[
\deg(\mtt t) - \deg(\mtt t^{\lda}) = \deg(\psi_{r_1} \ldots \psi_{r_m} e(\mbf i^{\lda}) ),
\] 
where $\mbf i^{\lda} \in I^d$ is the residue sequence of $\mtt t^{\lda}$.
Due to~\eqref{eq:algeven}, it follows that $\deg(\mtt t) \equiv \deg(\mtt t^{\lda}) \pmod 2$ for any
$\mtt t\in \scr T(\lda)$. Hence, by Lemma~\ref{lem:tlda}, $\e(\lda)$ is the parity of $\deg(\mtt t)$. The result follows. 
\end{proof}

The simple $H_d^{\Lda}$-modules $D(\lda)$ are indexed by the set $\scr{RP}^{\k}_d$ of \emph{$\k$-restricted multipartitions}, which is a certain subset of $\scr{P}^{\k}_d$ (see~\cite{BrundanKleshchev2009a, Kleshchev2010}). 
If $\Lda=\Lda_0$ then $\scr{RP}^{\k} = \scr{RP}^{(0)}$ is precisely the set of \emph{$2$-restricted} partitions, i.e.\ partitions $\lda=(\lda_1,\ldots,\lda_n)$ such that $\lda_r-\lda_{r+1}\le 1$ for all $r$. In the general case, the only known combinatorial description of $\scr{RP}^{\k}_d$ is a recursive one. 

By~\cite{BrundanKleshchev2009, HuMathas2010}, each simple module $D(\lda)$ is uniquely graded in such a way that it is isomorphic (as a graded module) to its dual with respect to the anti-automorphism of $H^{\Lda}_d$ fixing each of the generators~\eqref{eq:gens}. As a consequence, 
$\qdim D(\lda) = \ol{\qdim D(\lda)}$, where we write
$\bar{f} (q) = f(q^{-1})$ for any $f\in \mZ[q,q^{-1}]$. Moreover, it is known that $D(\lda)$ is isomorphic to the head of $S(\lda)$ as a graded module if the gradings are as described
(see \cite[Theorem 5.10]{BrundanKleshchev2009} or~\cite{HuMathas2010}).
Hence, Theorem~\ref{thm:1} implies the following result. 
 
\begin{cor}\label{cor:0}
 For any $\lda\in \scr{RP}_d^{\k}$, we have $\qdim D(\lda) (u) \in \mN u^{\e(\lda)}$.
\end{cor}

If the field $F$ is chosen so that its characteristic is $p\ge 0$ and $\lda\in \scr{RP}^{\k}_d$, $\mu\in \scr{P}^{\k}_d$, let 
$d_{\lda\mu}^{(p)} = [S(\lda): D(\mu)]_q \in \mN[q,q^{-1}]$ be the corresponding 
\emph{graded decomposition number} (see~\cite[\S 2.4]{BrundanKleshchev2009}). 
Thus, by definition, $d_{\lda\mu}^{(p)} = \sum_{n\in \mZ} b_n q^n$ where $b_n$ is the 
multiplicity of $D(\mu)\lan n \ran$ in the graded composition series of $S(\lda)$; here, 
$D(\mu)\lan n \ran$ is the module obtained from $D(\mu)$ by shifting all degrees by $n$. 
Then 
$D^{(p)} = (d_{\lda\mu}^{(p)})_{\lda,\mu}$ 
is the graded decomposition matrix of $H^{\Lda}_d$. It is known that $D^{(p)}$ depends only on $p$, not on the choice of $F$ (see~\cite[Section 6]{BrundanKleshchev2009a}). 

By~\cite[Theorem 5.17]{BrundanKleshchev2009}, we have
\begin{equation}\label{eq:dec}
 D^{(p)} = D^{(0)} A^{(p)}
\end{equation}
for a certain \emph{graded adjustment matrix} 
$A^{(p)}=(a_{\lda\mu}^{(p)})_{\lda,\mu\in \scr{RP}^{\k}_d }$, and the entries $a_{\lda\mu}^{(p)}$ have nonnegative coefficients. Also, 
$\ol{a_{\lda,\mu}^{(p)}} = a_{\lda,\mu}^{(p)}$ for all $\lda,\mu$. 
Further, by Theorems 3.19 and 5.15 of~\cite{BrundanKleshchev2009}, we have
$d_{\lda\mu}^{(0)} \in \d_{\lda\mu} + q\mN[q]$.

\begin{cor}\label{cor:1}
 Let $\lda\in \scr{P}^{\k}_d$ and $\mu\in \scr{RP}_d^{\k}$. 
Let $p$ be either a prime or $0$. Then
\begin{enumerate}[(i)]
 \item $d_{\lda\mu}^{(p)} (u) \in \mN u^{\e(\lda)+\e(\mu)}$.
 \item If $\lda\in \scr{RP}^{\k}_d$, then $a_{\lda\mu}^{(p)} (u) \in \mN u^{\e(\lda)+\e(\mu)}$.
\end{enumerate}
\end{cor}

\begin{proof}
 (i) Write $\qdim D(\mu)(u) = m u^{\e(\mu)}$, so that $m>0$. 
It follows from the definition of $D^{(p)}$ that 
\[
 \qdim S(\lda) (u) = \sum_{\nu \in \scr{RP}_d^{\k}} d_{\lda\nu}^{(p)} (u) \qdim D(\nu) (u).
\]
Therefore, if $d_{\lda\mu}^{(p)} (u)$ has a non-zero coefficient in $u^{1+\e(\lda)+\e(\mu)}$, then $\qdim S(\lda) (u)$ has coefficient at least $m$ in $u^{1+\e(\lda)}$, contradicting Theorem~\ref{thm:1}.

(ii) If $a_{\lda\mu}^{(p)} (u)$ has a non-zero coefficient in $u^{1+\e(\lda)+\e(\mu)}$, then 
so does $d_{\lda\mu}^{(p)} (u)$: indeed, since $d_{\lda\lda}^{(0)} (u)=1$, Eq.~\eqref{eq:dec} gives
\[
 d_{\lda\mu}^{(p)} (u) \in a_{\lda\mu}^{(p)} (u) + \mN[u].
\]
So (ii) follows from (i). 
\end{proof}

We have the following consequence for ungraded adjustment matrices and ungraded dimensions of simple modules. 

\begin{cor}\label{cor:ungr}
Let $\lda,\mu\in \scr{RP}^{\k}_d$. 
 If $\e(\lda)\ne \e(\mu)$, then $a_{\lda\mu}^{(p)}(1)$ is even.
If $\e(\lda) =1$, then $\qdim D(\lda) (1)$ is even.
\end{cor}

\begin{proof}
If $\e(\lda)\ne \e(\mu)$ then $a_{\lda\mu}^{(p)}$ has no constant term by Corollary~\ref{cor:1}(ii). Since
$\ol{a_{\lda\mu}^{(p)}} = a_{\lda\mu}^{(p)}$, it follows that $a_{\lda\mu}^{(p)} (1)$ is even. 
The second assertion is proved by the same argument applied to the polynomial $\qdim D(\lda)$, using Corollary~\ref{cor:0}.
\end{proof}

Now we specialise to the case $\Lda = \Lda_0$ and $p=2$, so that $H^{\Lda}_d \cong F \fr S_d$. 

\begin{cor}\label{cor:notinZ} Assume that $\Lda=\Lda_0$.
 There exist pairs $(\lda,\mu)$ of $2$-restricted partitions such that 
$a_{\lda\mu}^{(2)} \notin \mZ$. 
\end{cor}

\begin{proof}
 If $\e(\lda)\ne \e(\mu)$ and $a_{\lda\mu}^{(2)} (1) \ne 0$, then, by Corollary~\ref{cor:1}, 
$a_{\lda\mu}^{(2)}(u)$ is a positive multiple of $u$, which implies that 
$a_{\lda\mu}^{(2)}\notin \mZ$.
 Checking the tables of ungraded adjustment matrices in~\cite[Appendix B]{Mathas1999}, we find the following examples of pairs $(\lda,\mu)$ satisfying the conditions that 
$\e(\lda)\ne \e(\mu)$ and $a_{\lda\mu}^{(2)} (1) \ne 0$:
\[
 \begin{split}
   \lda = (3, 2^2, 1), & \; \mu = (1^8); \\ 
   \lda=(3, 2^2, 1^2), & \; \mu =(1^9); \\
   \lda=(3^2, 2, 1^2), & \; \mu=(1^{10}); \\
   \lda = (5,2^2, 1), & \; \mu = (3,1^7).
 \end{split}
\]
In each of these cases, $a_{\lda\mu}^{(2)}(1) =2$. 
(There are no such pairs $(\lda,\mu)$ for $d\le 7$.) 
\end{proof}

\begin{remark}\label{rem:qneg}
 Whenever $a_{\lda\mu}^{(p)} \notin \mZ$, we have $d_{\lda\mu}^{(p)} \notin \mZ[q]$. This follows from~\eqref{eq:dec} and the positivity properties quoted just after that equation.  
\end{remark}

\begin{remark}\label{rem:values}
Both Andrew Mathas and Sin{\'e}ad Lyle (personal communications) have found arguments showing that,
in fact,  
\begin{equation}\label{eq:rem}
 d_{(3,2^2,1), (1^8)}^{(2)}  = a_{(3,2^2, 1), (1^8)}^{(2)}  = q+q^{-1}.
\end{equation}
The first equality immediately follows from~\eqref{eq:dec} 
together with the known values of $D^{(0)}$ and the
ungraded version of $A^{(2)}$, as given in~\cite[Appendix B]{Mathas1999}. 
We reproduce Mathas's proof of the second equality here. 
By the proof of Corollary~\ref{cor:notinZ}, $a_{(3,2^2,1), (1^8) }  = q^m +q^{-m}$ 
for some $m>0$.
Assume that $\Char F=2$, and let $\mbf i = (0,1,0,1,0,1,0,1)$.
Since the trivial module $D(1^8)$ satisfies $e(\mbf i) D(1^8) = D(1^8)$, the graded vector space
$e(\mbf i) S(3,2^2,1)$ must have a non-zero component of degree $-m$. 
Recall that $\{ z_{\mtt t} \}_{\mtt t\in \scr{T}(3,2^2,1) }$ is a basis of $S(3,2^2,1)$. If $\mtt t$ is a standard tableau, we have (\cite[Lemma 4.4]{BrundanKleshchevWang2011})
\[
 e(\mbf i) z_{\mtt t} = \begin{cases}
                         z_{\mtt t} &\text{if } \mtt t \text{ has residue sequence } \mbf i,\\
                         0 & \text{otherwise.}
 \end{cases}
\]
 The only standard $(3,2^2,1)$-tableaux with residue sequence $\mbf i$ are as follows.
\[
 \young(123,47,58,6) \qquad
 \young(145,27,38,6) \qquad 
 \young(147,25,36,8) \qquad
 \young(147,25,38,6)
\]
These have degrees $1,1,1,-1$ respectively. It follows that $m=1$, so~\eqref{eq:rem} holds.
 
By a similar argument one shows that
\[
 a_{(3,2^2,1^2), (1^9)}^{(2)} = a^{(2)}_{(3^2,2, 1^2), (1^{10})} = q+q^{-1}.
\]
\end{remark}

\begin{remark}
 Mathas has extended Corollary~\ref{cor:notinZ} by showing that the entries of the graded adjustment matrices of $H^{\Lda_0}_{d}$ in quantum characteristic $3$ defined over a field of characteristic $2$ do not all belong to $\mZ$; see~\cite[Example 3.7.13]{Mathas2013}.
\end{remark}

\section*{Acknowledgements}
 The author is grateful to Sin{\'e}ad Lyle and Andrew Mathas for supplying proofs of the equalities in Remark~\ref{rem:values} and other helpful comments.


\begin{thebibliography}{1}

\bibitem{BrundanKleshchev2009a}
J.\ Brundan and A.\ Kleshchev,
\newblock Blocks of cyclotomic {H}ecke algebras and {K}hovanov-{L}auda
  algebras.
\newblock {\em Invent.\ Math.} \textbf{178} (2009), no.\ 3, 451--484.

\bibitem{BrundanKleshchev2009}
J.\ Brundan and A.\ Kleshchev,
\newblock Graded decomposition numbers for cyclotomic {H}ecke algebras.
\newblock {\em Adv.\ Math.} \textbf{222} (2009), no.\ 6, 1883--1942.

\bibitem{BrundanKleshchevWang2011}
J.\ Brundan, A.\ Kleshchev, and W.\ Wang,
\newblock Graded {S}pecht modules.
\newblock {\em J.\ Reine Angew.\ Math.} \textbf{655} (2011), 61--87.

\bibitem{HuMathas2010}
J.\ Hu and A.\ Mathas,
\newblock Graded cellular bases for the cyclotomic
  {K}hovanov-{L}auda-{R}ouquier algebras of type {A}.
\newblock {\em Adv.\ Math.} \textbf{225} (2010), no.\ 2, 598--642.

\bibitem{Kac1990}
V.G.\ Kac,
\newblock {\em Infinite-dimensional {L}ie algebras}, 3 ed.
\newblock Cambridge University Press, Cambridge, 1990.

\bibitem{KhovanovLauda2009}
M.\ Khovanov and A.D.\ Lauda,
\newblock {A} diagrammatic approach to categorification of quantum groups.\
  {I}.
\newblock {\em Represent. Theory} \textbf{13} (2009), 309--347.

\bibitem{Kleshchev2010}
A.\ Kleshchev,
\newblock Representation theory of symmetric groups and related {H}ecke
  algebras.
\newblock {\em Bull.\ Amer.\ Math.\ Soc.}, \textbf{47} (2010), 419--481.

\bibitem{Mathas1999}
A.\ Mathas,
\newblock {\em {I}wahori-{H}ecke algebras and Schur algebras of the symmetric
  group}.
\newblock American Mathematical Society, Providence, RI, 1999.

\bibitem{Mathas2013}
A.\ Mathas,
\newblock {\em Cyclotomic quiver {H}ecke algebras of type $A$}.
\newblock {A}r{X}iv: 1310.2142v1 (2013). 

\bibitem{Rouquier2008}
R.\ Rouquier,
\newblock 2-{K}ac-{M}oody algebras.
\newblock {A}r{X}iv: 0812.5023 (2008).

\end{thebibliography}

\end{document}